\documentclass[final,5p,times,twocolumn,authoryear]{elsarticle}

\usepackage{mathrsfs,amsmath,amssymb,amsthm,graphicx,color,dsfont}
\usepackage{multirow,xcolor,appendix}

\newtheorem{assumption}{Assumption}[section]

\newtheorem{theorem}{Theorem}
\newtheorem{remark}{Remark}[section]
\newtheorem{problem}{Problem}
\newtheorem{definition}{Definition}
\newtheorem{corollary}{Corollary}
\newtheorem{proposition}{Proposition}

\newcommand{\bremark}{\begin{remark} \begin{rm} }
\newcommand{\eremark}{ \end{rm} \rule{1mm}{2mm} \end{remark}}

\renewcommand{\t}{\mbox{\tiny\sf T}}
\newcommand{\R}{\mathbb{R}}

\begin{document}

\begin{frontmatter}

\title{\LARGE\bf Network games with dynamic players: Stabilization and output convergence to Nash equilibrium} 

\author{Meichen Guo}\ead{meichen.guo@rug.nl}
\author{Claudio De Persis}\ead{c.de.persis@rug.nl}

\address{ENTEG, Faculty of Science and Engineering, University of Groningen, Nijenborgh 4, 9747 AG Groningen, The Netherlands}


\begin{abstract}
  This paper addresses a class of network games played by dynamic agents using their outputs. Unlike most existing related works, the Nash equilibrium in this work is defined by functions of agent outputs instead of full agent states, which allows the agents to have more general and heterogeneous dynamics and maintain some privacy of their local states. The concerned network game is formulated with agents modeled by uncertain linear systems subject to external disturbances. The cost function of each agent is a linear quadratic function depending on the outputs of its own and its neighbors in the underlying graph. The main challenge stemming from this game formulation is that merely driving the agent outputs to the Nash equilibrium does not guarantee the stability of the agent dynamics. Using local output and the outputs from the neighbors of each agent, we aim at designing game strategies that achieve output Nash equilibrium seeking and stabilization of the closed-loop dynamics. Particularly, when each agents knows how the actions of its neighbors affect its cost function, a game strategy is developed for network games with digraph topology. When each agent is also allowed to exchange part of its compensator state, a distributed strategy can be designed for networks with connected undirected graphs or connected digraphs.
\end{abstract}

\end{frontmatter}

\thispagestyle{empty}

\section{Introduction}

Game theory has various applications to the control of multi-agent systems including smart grids, optical networks, and mobile sensor networks; see for example, \citet{Wang2014game_smartgrid,Stankovic2012,Pavel2006_opticalnetwork}. In these applications, the agents try to minimize local cost functions that depend on actions of their own and other players. The aim of the game is often set to seek a Nash Equilibrium (NE), \emph{i.e.}, no agent can gain by unilaterally changing its strategy. As far as we know, all NE seeking problems in literature use full agent states as decision variables and the sole purpose of the game strategy is to drive all the components of each agent state to the NE. These game theoretic problems have undesirable characters including the lack of privacy among agents and the restrictions on agent dynamics. First, as full agent states are used as decision variables, each agent has to know, or in the case of limited communication, observe the full states of all other agents. In this setting, it is impossible for the agents to converge to the NE while keeping some parts of their states unknown to the others. Second, in most existing works, the agents do not have independent dynamics such as in \citet{Ye2017TAC_NEseeking}, or they can have simple homogeneous dynamics such as in \cite{Romano}.
The recent work of \citet{Romano2019ECC} develops distributed NE seeking strategies for a class of heterogenous linear systems and defines the NE using partial local state. However, the agent dynamics therein take the special form of multi-integrators and at the defined NE, the  part of the local state that does not explicitly appear in the NE must be zero.
In summary, the existing NE seeking strategies are not applicable to engineering problems of multi-agent systems with general linear local dynamics. Motivated by these drawbacks in the existing works, we aim at solving an NE seeking problem such that, (i) if part of the local agent state is not directly involved in decision making, it can remain private from other agents; (ii) the agents can have more general and heterogeneous dynamics. However, it should be pointed out that this game formulation brings a major challenge that the outputs converging to the NE does not imply the stability of each agent. As we also assume that the agent states are not measurable, our goal is to design output feedback strategies that serve two purposes: stabilizing local dynamics and driving the outputs to the NE.
\\
\indent One important feature of NE seeking strategies in reference such as \citet{Lin2014CDC_formation,Parise2015_CDCnetworkgames,Koshal2016, Salehisadaghiani2016,Ye2017TAC_NEseeking,Deng2019Auto_EL, DePersis2019Auto_NEseeking,Romano} is that the designed strategies are distributed with respect to the underlying communication graph. The distributed game strategies proposed in \citet{Salehisadaghiani2016,Romano} exploit the underlying communication network topology so that each agent estimates the actions of others using information from its neighbors. Nonetheless, in general games, each agent needs information on the actions of all other agents to determine its own action. When the size of the network is large, the computational burden of each agent can be extremely heavy. This limitation leads to research on network games. In \citet{Parise2015_CDCnetworkgames}, distributed iterative strategies are proposed for agents to converge to an NE in network games with quadratic cost and convex constraints. \citet{Grammatico2018TCNS_networkgames} investigates the convergence of different equilibrium seeking proximal dynamics for multi-agent network games with convex cost functions, time-varying communication graph, and coupling constraints. Our work focuses on NE seeking in a class of network games, since in cooperative control of large networks, it can be computational difficult or unrealistic for each agent to consider the actions of all others. Specifically, the network games considered in this work is a class of quasi-aggregative games that each agent minimizes its cost function depending on its own actions and the actions of its neighbors in the underlying network topology. If the network is connected by a complete graph, the concerned network games become general or aggregative games such as the ones studied in \citet{Koshal2016,Salehisadaghiani2016,Deng2019Auto_EL, DePersis2019Auto_NEseeking, Romano}.

Another practical and important consideration in game theoretical engineering problems is the influence of external disturbances. In engineering applications, agents are often subject to external disturbances, for instance, the energy consumption demand in \citet{Wang2014game_smartgrid} and wind pushing the mobile robots in formation control in \citet{Romano}. However, designing NE seeking strategies capable of disturbance rejection has not gained much attention other than a few works such as \citet{Stankovic2012,Romano}. In this paper, the game is formulated with agents subject to deterministic disturbances generated by linear exosystems, which can be a combination of step functions and finitely many sinusoidal functions.

For a general or aggregative game setting, some recent works have investigated distributed NE seeking problems without or with simply local agent dynamics. \citet{Ye2017TAC_NEseeking} considers the NE seeking problem for agents communicating through an undirected and connected graph. Based on a consensus protocol and the gradient play approach, distributed NE seeking strategies are proposed for games with quadratic and nonquadratic cost functions. In the recent work \citet{Romano}, dynamic NE seeking strategies are proposed for single and double integrator networks subject to external disturbance modeled by linear exosystems. The cost function of each agent is a general convex cost function that depends on the actions of all the agents in the network. Under the assumption that the underlying communication graph is undirected and connected, dynamic strategies can be developed to estimate the actions of others and for the agents to converge to the unique Nash equilibrium. Reference \citet{Deng2019Auto_EL} investigates an aggregative game of Euler-Lagrange systems in the presence of uncertain parameters. The network topology is assumed to be undirected and connected, and the agents are free of external disturbances. As Euler-Lagrange systems are nonlinear second-order systems, the game strategies in \citet{Deng2019Auto_EL} are proposed based on double integrators.

The distributed robust NE seeking strategies proposed in this work are inspired by methods in linear output regulation. Compared with existing works on distributed NE seeking strategy design, this paper has the following contributions. First, we define the NE using outputs of general linear systems as decision variables, which leads to strategies having a wider range of engineering applications. Second, the designed NE seeking strategies are capable of not only outputs convergence to the NE but also stabilization of agent dynamics despite model uncertainties and external disturbances. Last but not least, the proposed strategies can be applied to networks with communication digraphs, which is not the case for most existing game strategies including the ones in \citet{Ye2017TAC_NEseeking,Romano,Deng2019Auto_EL}.

The rest of the paper is arranged as follows. The network game formulation and some preliminaries are given in Section~\ref{section_formulation}. The distributed NE seeking strategies are developed and analyzed in Section~\ref{section_mainresults}. In Section~\ref{section_simulation}, simulation results of the control of sensor networks are presented. Finally, some conclusive remarks are given in Section~\ref{section_conclusion}.

\emph{Notations.} $\R$ denotes the set of real numbers. $I_{n}$ denotes the $n\times n$ identity matrix. For vectors $x_{1},\dots,x_{N}$, $\text{col}(x_{1},\dots,x_{N})$ denotes $[x_{1}^{\t},\dots,x_{N}^{\t}]^{\t}$, and $x_{-i}$ denotes $[x_{1}^{\t},\dots,x_{i-1}^{\t},x_{i+1}^{\t},\dots,x_{N}^{\t}]^{\t}$. Given a matrix $A$, $\text{vec}(A)=[A_{1}~\dots~A_{l}]$ with $A_{i}$ being the $i$th row of $A$. Given matrices $A_{1},\dots,A_{N}$, $\text{blockdiag}(A_{1},\dots,A_{N})$ denotes the block diagonal matrix with $A_{i}$  on the diagonal.

\section{Game formulation and preliminaries}\label{section_formulation}

In this section, the formulation of the NE seeking problem for a class of network games will be presented. Then, we give some preliminaries on linear output regulation that will be used in the subsequent sections. It is shown that any controller that solves the output regulation problem of the reformulated agent dynamics also solves the NE seeking problem.

\subsection{Game formulation}
Consider a game denoted by $\mathcal{G}(\mathcal{I},J_{i},\Psi_{i})$ played by $N$ agents having dynamics
\begin{align}\label{playerdynamics_i}
\dot{x}_{i}&=A_{i}(\mu_{i})x_{i}+B_{i}(\mu_{i})u_{i}+P_{i}(\mu_{i})w_{i},\notag\\
y_{i}&=C_{i}(\mu_{i})x_{i},
\end{align}
for each $i\in\mathcal{I}:=\{1,\dots,N\}$, where $w_{i}\in\Theta_{i}\subseteq\R^{q_{i}}$ is the disturbance, $x_{i}\in\Phi_{i}\subseteq\R^{n_{i}}$ the state, $u_{i}\in\Omega_{i}\subseteq\R^{m_{i}}$ the input and strategy vector, $y_{i}\in\Psi_{i}\subseteq\R^{p_{i}}$ the output and decision variable, and denote $\mu\in\mathbb{R}^{n_{i}(n_{i}+m_{i}+q_{i}+{p_{i}})}$ as the uncertainty. The disturbance $w_{i}$ is generated by
\begin{align}\label{exosystem}
\dot{w}_{i}&=S_{i}w_{i},
\end{align}
where the initial value $w_{i}(0)\in\Theta_{i}\subseteq\R^{q_{i}}$ with $\Theta_{i}$ being forward-invariant. The uncertain matrices $A_{i}(\mu_{i})$, $B_{i}(\mu_{i})$, $P_{i}(\mu_{i})$, and $C_{i}(\mu_{i})$, $i\in\mathcal{I}$, are in the form of
\begin{align*}
A_{i}(\mu_{i})&=A_{i}+\Delta A_{i},~~B_{i}(\mu_{i})=B_{i}+\Delta B_{i},\\
P_{i}(\mu_{i})&=P_{i}+\Delta P_{i},~~C_{i}(\mu_{i})=C_{i}+\Delta C_{i},
\end{align*}
where $A_{i}$, $B_{i}$, $P_{i}$, $C_{i}$ are the known nominal parts and $\Delta A_{i}$, $\Delta B_{i}$, $\Delta P_{i}$, $\Delta C_{i}$ are the uncertain parts. The uncertainty for each agent $i$, $i\in\mathcal{I}$ can be written as $\mu_{i}=[\text{vec}(\Delta A_{i})~\text{vec}(\Delta B_{i})~\text{vec}(\Delta P_{i})~\text{vec}(\Delta C_{i})]^{\t}$. 

Define $G_{c} = (\mathcal{I},\mathcal{E})$ as the underlying communication graph of the network where $\mathcal{I}$ is the index set as defined above and $\mathcal{E}\subset \mathcal{I} \times \mathcal{I}$ is the edge set. There is an edge between nodes $i$ and $j$ if $(i,j) \in \mathcal{E}$, $i, j \in\mathcal{I}$. For an undirected graph, if $(i,j) \in \mathcal{E}$ then $(j,i) \in \mathcal{E}$. Denote the neighbor set of agent $i$ as $\mathcal{N}_{i} \subset \mathcal{I}$. For an undirected graph, if there exists a path between every pair of nodes, the undirected graph is connected. A digraph is weakly connected if there exists an undirected path between every pair of nodes, and is strongly connected if there exists a directed path between every pair of nodes.

Now, we give the definition of NE using outputs of (\ref{playerdynamics_i}) as decision variables in the network game $\mathcal{G}(\mathcal{I},J_{i},\Psi_{i})$. In the game, each agent tries to minimize a local cost function denoted by $J_{i}$, $i\in\mathcal{I}$.

\begin{definition}[NE in network games]\label{definition_NE} 
  Given a network game $\mathcal{G}(\mathcal{I},J_{i},\Psi_{i})$, the output $y^{*}:=\text{col}(y_{1}^{*},\dots,y_{N}^{*})$ is an NE of $\mathcal{G}$ if
  \begin{align*}
    J_{i}(y_{i}^{*},y_{\mathcal{N}_{i}}^{*})\le J_{i}(y_{i},y_{\mathcal{N}_{i}}^{*}).
  \end{align*}
\end{definition}

\bremark(\emph{NE in network and general games})
  Definition~\ref{definition_NE} becomes the same as the definition of NE in general or aggregative games such as \citet{Koshal2016,Salehisadaghiani2016,Deng2019Auto_EL, DePersis2019Auto_NEseeking, Romano} if $R_{ij}\ne 0$, $Q_{ij}\ne0$ for all $i,j\in\mathcal{I}$, \emph{i.e.}, $G_{c}$ is a complete graph.
\eremark

The local cost function concerned in this work is defined as
\begin{align}\label{localcostfunction}
J_{i}(y_{i},y_{-i})=y_{i}^{\t}R_{ii}y_{i}+Q_{ii}y_{i}+q_{i}+\sum_{j\in\mathcal{I},\;j\ne i}(y_{i}^{\t}R_{ij}y_{j}+y_{j}^{\t}Q_{ij}y_{j})
\end{align}
where $R_{ii}\in\mathbb{R}^{p_{i}\times p_{i}}$, $R_{ii}>0$ for $i\in\mathcal{I}$, $Q_{ii}\in\mathbb{R}^{1\times p_{i}}$, $q_{i}\in\R$, $R_{ij}\in\mathbb{R}^{p_{i}\times p_{j}}$, and $Q_{ij}\in\mathbb{R}^{p_{j}\times p_{j}}$ for $i,j\in\mathcal{I},~j\ne i$. Denote the neighbor set of agent $i$ as $\mathcal{N}_{i}$. Assume that $R_{ij}\ne 0$, $Q_{ij}\ne0$ for $j\in\mathcal{N}_{i}$, and $R_{ij}=0$, $Q_{ij}=0$ otherwise. Then, the cost function can also be written as $J_{i}(y_{i},y_{\mathcal{N}_{i}})$.

Note that since $R_{ii}>0$, for each $i\in\mathcal{I}$, $\Psi_{i}=\R^{p_{i}}$ and $\Phi_{i}=\R^{n_{i}}$, the local cost function $J_{i}$ is strictly convex and radially unbounded in $y_{i}$ for all $y_{\mathcal{N}_{i}}\in\Psi_{\mathcal{N}_{i}}$. Then according to \citet[Corollary 4.2]{Bacsar1999Book}, there exists an NE for network game $\mathcal{G}(\mathcal{I},J_{i},\Psi_{i})$.

Denote the partial gradient of cost function $J_{i}$ as $\nabla_{i}J_{i}(y_{i},y_{\mathcal{N}_{i}})=(\frac{\partial J_{i}}{\partial y_{i}})^{\t}$ and the pseudo-gradient as $F(y)=\text{col}(\nabla_{1}J_{1}(y_{1},y_{\mathcal{N}_{1}}),\dots,\nabla_{N}J_{N}(y_{N},y_{\mathcal{N}_{N}}))$. Then, the pseudo-gradient can be written as $F(y)=\bar{R}y+\bar{Q}$ where
\begin{align*}
\bar{R}=
  \begin{bmatrix}
    R_{11}+R_{11}^{\t} & R_{12} & \cdots & R_{1N} \\
    R_{21} & R_{22}+R_{22}^{\t} & \cdots & R_{2N} \\
    \vdots & \vdots & \ddots & \vdots \\
    R_{N1} & R_{N2} & \cdots & R_{NN}+R_{NN}^{\t}
  \end{bmatrix},~~
\bar{Q}=
  \begin{bmatrix}
    Q_{11}^{\t} \\
    Q_{22}^{\t} \\
    \vdots \\
    Q_{NN}^{\t}
  \end{bmatrix}.
\end{align*}

\begin{assumption}\label{assumption_monotone}
  The matrix $\bar{R}$ is positive definite.
\end{assumption}


Following 
\citet[Theorem 2.3.3]{Facchinei2003}, under Assumption~\ref{assumption_monotone}, the mapping $F$ is strictly monotone, and the game has a unique NE $y^{*}$ satisfying $F(y^{*})=\bar{R}y^{*}+\bar{Q}=0$.


\begin{remark}[Monotonicity of $F$]
Different assumptions on the monotonicity of mapping $F$ have been used in existing works to guarantee the uniqueness of the NE. For example, for general cost functions, \citet{Koshal2016,Salehisadaghiani2016} assume the mapping to be strictly monotone, and \citet{Deng2019Auto_EL,DePersis2019Auto_NEseeking} assume that the mapping is strongly monotone. Note that for general cost functions, the assumption of strict monotonicity is weaker than that of strong monotonicity. In this work, as the cost function of each agent is in the linear quadratic form (\ref{localcostfunction}), the strict monotonicity and the strong monotonicity of mapping $F$ are equivalent.
\end{remark}

Another assumption is made to exclude the trivial case where the disturbance $w_{i}$ for $i\in\mathcal{I}$ exponentially decays to $0$.

\begin{assumption}\label{assumption_exosyst}
  For each $i\in\mathcal{I}$, $S_{i}$ has no eigenvalues with negative real parts.
\end{assumption}

Now we give the formulation of the NE seeking problem for the network game $\mathcal{G}$.

\begin{problem}\label{problem_NEseeking}
Under Assumptions~\ref{assumption_monotone} to \ref{assumption_exosyst}, for the network game $\mathcal{G}(\mathcal{I},J_{i},\Psi_{i})$, use local partial gradient $e_{i}=\nabla_{i}J_{i}(y_{i},y_{\mathcal{N}_{i}})$, $i\in\mathcal{I}$ to design strategy $u_{i}$ such that

(i) the closed-loop dynamics of all agents are stable, and

(ii) the output $y:=\text{col}(y_{1},\dots,y_{N})$ converge to the NE $y^{*}$ that satisfies $\bar{R}y^{*}+\bar{Q}=0$.
\end{problem}

In the contrary to most existing works on NE seeking, the solution of Problem~\ref{problem_NEseeking} does not only need to drive the decision variables to the NE, but also guarantee the stability of the closed-loop systems. This objective is similar to output regulation problems where the aim of the regulator is to achieve reference tracking and/or disturbance rejection while guaranteeing that the closed-loop system is stable. In fact, we will get inspirations from output regulation for game strategies design in Section \ref{section_mainresults}.

\subsection{Preliminaries}
First, we write the dynamics of all the agents in a stacked form. Denote
\begin{align*}
  w&=\text{col}(w_{1},\dots,w_{N}),~~ x=\text{col}(x_{1},\dots,x_{N}),\\
  u&=\text{col}(u_{1},\dots,u_{N}),~~~\mu=\text{col}(\mu_{1},\dots,\mu_{N}),
\end{align*}
$\bar{m}=\sum_{i\in\mathcal{I}}m_{i}$, $\bar{n}=\sum_{i\in\mathcal{I}}n_{i}$, $\bar{p}=\sum_{i\in\mathcal{I}}p_{i}$, and $\bar{q}=\sum_{i\in\mathcal{I}}q_{i}$. Use the pseudo-gradient as the regulated error, {\it i.e.}, $e=F(y)=\bar{R}y+\bar{Q}$.

In order to write the regulated error $e$ as a function of the output $y$ and an exogenous signal, we construct a linear exosystem that consists of the disturbance generator (\ref{exosystem}) and a constant component. Specifically, for $i\in\mathcal{I}$, define $v_{i}=\text{col}(v_{i1},v_{i2})$ with initial condition $v_{i}(0)=\text{col}(w_{i}(0),1)$ satisfying
\begin{align}\label{exosystem_new}
 \dot{v}_{i}=\widetilde{S}_{i}v_{i}
 =\begin{bmatrix}
           S_{i} & 0 \\
           0 & 0
         \end{bmatrix}v_{i}.
\end{align}

Then, the dynamics of the $i$th agent can be written as
\begin{align}\label{playerdynamics_i_new}
\dot{x}_{i}&=A_{i}(\mu_{i})x_{i}+B_{i}(\mu_{i})u_{i} +\widetilde{P}_{i}(\mu_{i})v_{i},\notag\\
\dot{v}_{i}&=\widetilde{S}_{i}v_{i},\notag\\
e_{i}&=\bar{F}_{i}(x,v_{i}),
\end{align}
where $\widetilde{P}_{i}(\mu_{i})=[P_{i}(\mu_{i})~~0]$. The local regulated error $e_{i}$ is defined as the local partial gradient described by
\begin{align}
\bar{F}_{i}(x,v_{i})=(R_{ii}+R_{ii}^{\t})C_{i}(\mu_{i})x_{i}+\sum_{j\in\mathcal{N}} R_{ij}C_{j}(\mu_{i})x_{j}+\widetilde{Q}_{ii}v_{i}
\end{align}
with $\widetilde{Q}_{ii}=[0_{1\times q_{i}}~~Q_{ii}^{\t}]$.

\bremark(\emph{Extended exosystem})
  In exosystem (\ref{exosystem_new}), $v_{i1}$ is the same as $w_{i}$ in (\ref{exosystem}) and $v_{i2}$ is constant $1$. Then, we can use $Q_{ii}^{\t}v_{i2}$ to replace $Q_{ii}^{\t}$ in the expression of the local partial gradient $e_{i}$. It should be pointed out that, adding a constant component $v_{i2}$ is necessary and does not complicate the strategy design. If (\ref{exosystem}) already has a constant component, \emph{i.e.} $S_{i}$ has an eigenvalue at $0$, the addition of $v_{i2}$ will not change the subsequent design of the internal model. On the other hand, if $S_{i}$ has no eigenvalue at $0$, it is necessary to take the constant $v_{i2}$ into consideration. We will show later that for the special case where the agents are not subject to external disturbances, due to the constant vector $Q_{ii}$ in the cost function, the controller needs to contain an integrator, which is the internal model for constant exogenous signals.
\eremark

We denote $(\mathbf{x}_{i},\mathbf{u}_{i})$ for each $i\in\mathcal{I}$ as the steady state of (\ref{playerdynamics_i_new}) such that
\begin{align}\label{steady_state_equation}
0 & = A_{i}(\mu_{i})\mathbf{x}_{i}+B_{i}(\mu_{i})\mathbf{u}_{i} +\widetilde{P}_{i}(\mu_{i})v_{i},\notag\\
0 & = \bar{F}_{i}(\mathbf{x},v_{i}),
\end{align}
and the corresponding steady-state output as $\mathbf{y}_{i}=C_{i}(\mu_{i})\mathbf{x}_{i}$ where $\mathbf{x}=\text{col}(\mathbf{x}_{1},\dots,\mathbf{x}_{N})$. Then, seeing $e_{i}$ as the local regulated error, the distributed cooperative output regulation problem of (\ref{playerdynamics_i_new}) is to design a controller using $e_{i}$ such that $(x_{i},u_{i})$ converge to $(\mathbf{x}_{i},\mathbf{u}_{i})$ for all $i\in\mathcal{I}$. In fact, the solution to the cooperative output regulation problem also solves the NE seeking Problem~\ref{problem_NEseeking}.

\begin{proposition}\label{proposition}
Under Assumptions~\ref{assumption_monotone} to \ref{assumption_exosyst}, if there exists $(\mathbf{x}_{i},\mathbf{u}_{i})$ such that (\ref{steady_state_equation}) holds for all $i\in\mathcal{I}$, the corresponding $\mathbf{y}=\text{col}(\mathbf{y}_{1},\dots,\mathbf{y}_{N})$ is the NE of the network game $\mathcal{G}(\mathcal{I},J_{i},\Psi_{i})$.
\end{proposition}

Now, our aim is to design distributed controllers for cooperative output regulation of (\ref{playerdynamics_i_new}). To this end, some standard assumptions in the framework of linear output regulation will be given first. In what follows, we denote
\begin{align*}
\bar{A}(\mu)& = \text{blockdiag}(A_{1}(\mu_{1}),\dots,A_{N}(\mu_{N})), \\
\bar{B}(\mu)& = \text{blockdiag}(B_{1}(\mu_{1}),\dots,B_{N}(\mu_{N})), \\
\bar{C}(\mu)& = \text{blockdiag}(C_{1}(\mu_{1}),\dots,C_{N}(\mu_{N})), \\
\bar{P}(\mu)& = \text{blockdiag}(\widetilde{P}_{1}(\mu_{1}),\dots,\widetilde{P}_{N}(\mu_{N})),\\
\widehat{Q}& = \text{blockdiag}(\widetilde{Q}_{11},\dots,\widetilde{Q}_{NN}).
\end{align*}
Also, the nominal parts of $\bar{A}(\mu)$, $\bar{B}(\mu)$, $\bar{C}(\mu)$, and $\bar{P}(\mu)$ are denoted by $\bar{A}$, $\bar{B}$, $\bar{C}$, and $\bar{P}$, respectively.

\begin{assumption}\label{assumption_stabilizability&detectability}
  For each $i\in\mathcal{I}$, $(A_{i},B_{i})$ is stabilizable and $(A_{i},C_{i})$ is detectable.
\end{assumption}
\bremark(\emph{Overall network stabilizability and detectability})
  As matrices $R_{ii}$ are positive definite for all $i\in\mathcal{I}$, the detectability of pair $(A_{i},C_{i})$ is equivalent to the detectability of the pair $(A_{i},(R_{ii}+R_{ii}^{\t})C_{i})$. Moreover, since the dynamics of the agents are decoupled, the pair $(\bar{A},\bar{B})$ is stabilizable, $(\bar{A},\bar{C})$ and $(\bar{A},\bar{R}\bar{C})$ are detectable.
\eremark

\begin{assumption}\label{assumption_minimumphase}
  For all $\lambda\in\text{spec}(S_{i})\bigcup\{0\}$, $i\in\mathcal{I}$,
  \begin{align}\label{equation_miniphaserank}
    \text{rank}
    \begin{bmatrix}
      A_{i}-\lambda I & B_{i} \\
      C_{i} & 0
    \end{bmatrix}=n_{i}+p_{i}.
  \end{align}
\end{assumption}

\bremark(\emph{Existence of steady state})
Assumption~\ref{assumption_minimumphase} guarantees the existence of the steady state $(\mathbf{x}_{i},\mathbf{u}_{i})$ \citet[Theorem 1.9]{Huang2004}. The appendix presents some conditions guaranteed by Assumption~\ref{assumption_minimumphase} that will be used in subsequent sections.
\eremark

To reject the disturbance generated by exosystem (\ref{exosystem_new}) and handle the uncertainties in the agent dynamics, an internal model is constructed for each agent $i$, $i\in\mathcal{I}$. The following is a general definition of an internal model given an exosystem in the form $\dot{w}=Sw$.

\begin{definition}[Internal model]\citet[Definition 1.25]{Huang2004}\label{definition_IM} 
  Given any square matrix ${S}$, a pair of matrices $(M_{1},M_{2})$ incorporates a $p$-copy internal model of ${S}$ if
  \begin{align*}
    M_{1}=V\begin{bmatrix}
             T_{1} & T_{2} \\
             0 & G_{1}
           \end{bmatrix}V^{-1},~~
    M_{2}=V\begin{bmatrix}
             T_{3} \\
             G_{2}
           \end{bmatrix}
  \end{align*}
  where $(T_{1},T_{2},T_{3})$ are arbitrary constant matrices of any compatible dimensions, $V$ is any nonsingular matrix with the same dimension as $M_{1}$ and
  \begin{align*}
    G_{1}=\text{blockdiag}(\beta_{1},\dots,\beta_{p}), ~~G_{2}=\text{blockdiag}(\sigma_{1},\dots,\sigma_{p}),
  \end{align*}
  where $\beta_{i}$ are square matrices and $\sigma_{i}$ are column vectors with appropriate dimensions, $(\beta_{i},\sigma_{i})$ are controllable pairs, and the characteristic polynomials of $\beta_{i}$ are the same as the minimal polynomial of ${S}$ for all $i=1,\dots,p$.
\end{definition}

\bremark 
  By Definition~\ref{definition_IM}, as a special case of $(M_{1},M_{2})$, the pair of $(G_{1},G_{2})$ also incorporates a $p$-copy internal model of matrix $S$.
\eremark

\bremark (\emph{``p-copy'' Internal model})
Under Definition~\ref{definition_IM}, the dimension of the internal model is the dimension of the output times the order of the minimal polynomial of $S$. The internal model design in \citet{Isidori2003book} has the same dimension but does not use the term ``$p$-copy''.
\eremark

\section{Distributed game strategies}\label{section_mainresults}
In this section, two distributed output feedback control strategies are developed for the NE seeking problem. We first consider the case where the agents are connected by a diagraph without loops. Then, by letting the agents communicate more information with their neighbors, we relax the assumption on the information graph.

\subsection{Directed communication graph}
Using the internal model approach, a distributed error feedback strategy is designed in the form of
\begin{align}\label{controller_distributed1}
  \dot{\eta}_{i} & = M_{i1}\eta_{i}+M_{i2}e_{i},\notag\\
  u_{i} & = K_{i}\eta_{i},
\end{align}
where $\eta_{i}\in\mathbb{R}^{(n_{i}+p_{i}s_{i})}$, $K_{i}=[K_{i1}~~K_{i2}]$ is the gain matrix,
\begin{align*}
  M_{i1}=\begin{bmatrix}
    A_{i}+B_{i}K_{i1}-L_{i}(R_{ii}+R_{ii}^{\t})C_{i} & B_{i}K_{i2} \\
    0_{p_{i}s_{i}\times n_{i}} & G_{i1}
  \end{bmatrix},~~
  M_{i2}=\begin{bmatrix}
           L_{i} \\
           G_{i2}
         \end{bmatrix},
\end{align*}
$s_{i}$ is the order of the minimal polynomial of $\widetilde{S}_{i}$, $L_{i}$ is such that $A_{i}-L_{i}(R_{ii}+R_{ii}^{\t})C_{i}$ is Hurwitz, and $(G_{i1},G_{i2})$ incorporates a $p_{i}$-copy internal model of matrix $\widetilde{S}_{i}$. By Definition~\ref{definition_IM}, the pair $(M_{i1},M_{i2})$ also incorporates a $p_{i}$-copy internal model of matrix $\widetilde{S}_{i}$.

Before presenting the main result, we make the following assumption on the graph.
\begin{assumption}\label{assumption_graph}
  The communication graph among the agents is a digraph with no loop.
\end{assumption}

  In most existing works on distributed NE seeking strategy for games, for instance \citet{Salehisadaghiani2016,Ye2017TAC_NEseeking, Godjov2018TAC,Deng2019Auto_EL,DePersis2019Auto_NEseeking,Romano}, the network topology is assumed to be undirected and connected. As far as we know, the strategies proposed in the aforementioned works are not applicable to games with directed communication graphs. However, it is noted that the controller (\ref{controller_distributed1}) cannot solve NE seeking problem with undirected communication graphs. A distributed strategy that is able to handle both directed and undirected communication graphs will be presented in the next subsection by allowing the neighboring agents to exchange some additional information.

\bremark(\emph{State privacy in (\ref{controller_distributed1})}) \label{remark_privacy}
In controller (\ref{controller_distributed1}), each agent $i$, $i\in\mathcal{I}$, exchanges output $y_{i}$ with its neighbors. Note that its neighbors are not able to reconstruct the full state $x_{i}$ of agent $i$ only using this output. In the case where $C_{i}\ne I_{n_{i}}$ and $p_{i}\ne n_{i}$, at least part of the agent state can remain private from its neighbors.
\eremark

Now we are ready to present the main result of this subsection.

\begin{theorem}[Distributed strategy under communication digraphs]\label{theorem_digraph}
  Under Assumptions~\ref{assumption_stabilizability&detectability}, \ref{assumption_minimumphase}, and \ref{assumption_graph}, the distributed strategy (\ref{controller_distributed1}) is a solution to the NE seeking Problem~\ref{problem_NEseeking}.
\end{theorem}

\begin{proof}
Denote $z_{i}=\text{col}(x_{i},\eta_{i})$, and $z=\text{col}(z_{1},\dots,z_{N})$. The closed-loop system can be written as
\begin{align}\label{closed-loop syst}
  \dot{z} & = A_{c}(\mu)z+P_{c}(\mu)v,\notag\\
  \dot{v} & = \widehat{S}v,\notag\\
  e & = C_{c}(\mu)z+Q_{c}v
\end{align}
where $v=\text{col}(v_{1},\dots,v_{N})$ and $\widehat{S}=\text{blockdiag}(\widetilde{S}_{1},\dots,\widetilde{S}_{N})$.
We use $A_{c}$, $P_{c}$ and $C_{c}$ to denote the closed-loop system composed of the nominal dynamics and the controller (\ref{controller_distributed1}), where $A_{c}$ is a block matrix with diagonal blocks $A_{ci}$ and off diagonal blocks $E_{ij}$, $C_{c}$ is a block matrix with blocks $\widetilde{C}_{ij}$ for $i,j\in\mathcal{I}$, $i\ne j$, $P_{c}=\bar{P}$, $Q_{c}=\widehat{Q}$ and
\begin{align*}
A_{ci}&=
   \begin{bmatrix}
     A_{i} & B_{i}K_{i1} & B_{i}K_{i2} \\
     L_{i}(R_{ii}+R_{ii}^{\t})C_{i} & A_{i}+ B_{i}K_{i1}-L_{i}(R_{ii}+R_{ii}^{\t})C_{i} & B_{i}K_{i2} \\
     G_{i2}(R_{ii}+R_{ii}^{\t})C_{i} & 0_{v_{i}\times n_{i}} & G_{i1}
   \end{bmatrix},\\
E_{ij}&=
   \begin{bmatrix}
     0_{n_{i}\times n_{j}} & 0_{n_{i}\times n_{j}} & 0_{n_{i}\times v_{j}} \\
     L_{i}R_{ij}C_{j} & 0_{n_{i}\times n_{j}} & 0_{n_{i}\times v_{j}} \\
     G_{i2}R_{ij}C_{j} & 0_{v_{i}\times n_{j}} & 0_{v_{i}\times v_{j}}
   \end{bmatrix},\\
\widetilde{C}_{ij}&=
\begin{cases}
     [(R_{ii}+R_{ii}^{\t})C_{i} ~~ 0_{n_{i}\times n_{i}} ~~ 0_{n_{i}\times v_{i}}], & \mbox{if } i=j \\
     [R_{ij}C_{j} ~~ 0_{n_{i}\times n_{j}} ~~ 0_{n_{i}\times v_{j}}], & \mbox{otherwise}.
\end{cases}
\end{align*}
According to \citet[Theorem 1.31]{Huang2004}, under Assumptions~\ref{assumption_exosyst} and \ref{assumption_stabilizability&detectability}, there exists a dynamic output feedback controller in the form of (\ref{controller_distributed1}) such that the closed-loop system is stable and the error $e$ converges to $0$ asymptotically, if and only if matrix $A_{c}$ is Hurwitz, and the regulator equations
\begin{align}\label{regulator equations}
    Z\widehat{S} &= A_{c}(\mu)Z+P_{c}(\mu),\notag \\
    0 &= C_{c}(\mu)Z+Q_{c},
\end{align}
have a unique solution $Z$ for any $\mu$ in an open neighborhood of $\mu=0$.


First, we examine the stability of the nominal matrix $A_{c}$. Under Assumption~\ref{assumption_graph}, we can label the agents such that $i<j$ if $(i,j)\in\mathcal{E}$ for $i,j\in\mathcal{I}$. Then, $A_{c}$ becomes a block lower triangular matrix. For each $i\in\mathcal{I}$, the diagonal $A_{ci}$ is similar to the matrix
\begin{align*}
\bar{A}_{ci}=
  \begin{bmatrix}
    A_{i}+B_{i}K_{i1} & B_{i}K_{i1} & B_{i}K_{i2} \\
    0_{n_{i}\times n_{i}} & A_{i}-L_{i}(R_{ii}+R_{ii}^{\t})C_{i} & 0_{n_{i}\times v_{i}}\\
    G_{i2}(R_{ii}+R_{ii}^{\t})C_{i} & 0_{v_{i}\times n_{i}} & G_{i1}
  \end{bmatrix}.
\end{align*}
Note that there exists $L_{i}$ for $i\in\mathcal{I}$ such that $A_{i}-L_{i}(R_{ii}+R_{ii}^{\t})C_{i}$ is Hurwitz, as the pair $(A_{i},(R_{ii}+R_{ii}^{\t})C_{i})$ is detectable. Under Assumption~\ref{assumption_minimumphase} and by the definition of the internal model, we have that for all $\lambda\in\text{spec}(G_{i1})$,
\begin{align}\label{equation_minimunphase}
  \text{rank}\begin{bmatrix}
        A_{i}-\lambda I & B_{i} \\
        (R_{ii}+R_{ii}^{\t})C_{i} & 0_{p_{i}\times m_{i}}
      \end{bmatrix} =n_{i}+p_{i}.
\end{align}
Then, according to \citet[Lemma 1.26]{Huang2004}, under Assumptions~\ref{assumption_exosyst} and \ref{assumption_stabilizability&detectability}, the pair
\begin{align*}
  \Big{(}\begin{bmatrix}
        A_{i} & 0_{n_{i}\times v_{i}} \\
        G_{i2}(R_{ii}+R_{ii}^{\t})C_{i} & G_{i1}
      \end{bmatrix},
      \begin{bmatrix}
        B_{i} \\
        0_{v_{i}\times m_{i}}
      \end{bmatrix}\Big{)}
\end{align*}
is stabilizable. Thus, there exists control gain $[K_{i1}~K_{i2}]$ such that
\begin{align*}
  \begin{bmatrix}
    A_{i} & 0_{n_{i}\times v_{i}} \\
    G_{i2}(R_{ii}+R_{ii}^{\t})C_{i} & G_{i1}
  \end{bmatrix}+
  \begin{bmatrix}
        B_{i} \\
        0_{v_{i}\times m_{i}}
  \end{bmatrix}&
  \begin{bmatrix}
       K_{i1} & K_{i2}
  \end{bmatrix}\\
  =&\begin{bmatrix}
    A_{i}+B_{i}K_{i1} & B_{i}K_{i2} \\
    G_{i2}(R_{ii}+R_{ii}^{\t})C_{i} & G_{i1}
  \end{bmatrix}
\end{align*}
is Hurwitz. Then, for $i\in\mathcal{I}$, $\bar{A}_{ci}$, and consequently $A_{ci}$, are Hurwitz, which shows the diagonal blocks of $A_{c}$ are all Hurwitz. As $A_{c}$ is a block lower triangular matrix, it is also Hurwitz.

Next, we show there exists a unique solution $Z$ to the regulator equations (\ref{regulator equations}). As $A_{c}$ is Hurwitz, according to \citet[Lemma 1.27]{Huang2004}, the equations
\begin{align}\label{regulation equations nominal1}
X\widehat{S} &= \bar{A}X+\bar{B}\bar{K}\Xi+\bar{P},\notag\\
\Xi\widehat{S} &= \bar{M}_{1}\Xi+\bar{M}_{2}(\bar{R}\bar{C}X+\bar{Q}),\notag\\
0 &= \bar{R}\bar{C}X+\bar{Q}
\end{align}
has a unique solution $(X,\Xi)$ for any matrices $\bar{P}$ and $\bar{Q}$, where
\begin{align*}
\bar{K}&=\text{blockdiag}(K_{1},\dots,K_{N}),\\
\bar{M}_{1}&=\text{blockdiag}(M_{11},\dots,M_{N1}),\\
\bar{M}_{2}&=\text{blockdiag}(M_{12},\dots,M_{N2}),
\end{align*}
Note that the matrix equations (\ref{regulation equations nominal1}) can be put into the form of
\begin{align}\label{regulation equations nominal2}
    X_{c}\widehat{S} = A_{c}X_{c}+P_{c},\quad
    0 = C_{c}X_{c}+Q_{c},
\end{align}
and the solvability of (\ref{regulation equations nominal2}) means the solvability of regulator equations (\ref{regulator equations}) for any $\mu$ in any open neighborhood of $\mu=0$. Therefore, by \citet[Lemma 1.20]{Huang2004}, the distributed controller (\ref{controller_distributed1}) solves the output regulation problem. Then, following Proposition~\ref{proposition}, controller (\ref{controller_distributed1})  also solves the NE seeking Problem~\ref{problem_NEseeking}.
\end{proof}

For the case where the agents are free of disturbances or subject to constant external disturbances, the exosystem (\ref{exosystem_new}) satisfies $\text{spec}(\widetilde{S}_{i})=\{0\}$. Then, the controller (\ref{controller_distributed1}) can be simplified. The distributed output feedback controller for disturbance-free NE seeking problem can be derived directly from Theorem~\ref{theorem_digraph}.

\begin{corollary}\label{corollary1}
  Consider an NE seeking problem defined in Problem~\ref{problem_NEseeking} with agent dynamics
\begin{align}\label{playerdynamics_i_nodisturbance}
\dot{x}_{i}&=A_{i}(\mu_{i})x_{i}+B_{i}(\mu_{i})u_{i},\notag\\
y_{i}&=C_{i}(\mu_{i})x_{i},
\end{align}
and cost functions (\ref{localcostfunction}) for all $i\in\mathcal{I}$. Under Assumptions \ref{assumption_stabilizability&detectability}, \ref{assumption_minimumphase} and \ref{assumption_graph}, there exist matrices $L_{i}$, $K_{i1}$ and $K_{i2}$, such that the NE seeking problem has a solution in the form of (\ref{controller_distributed1}) with
\begin{align*}
  M_{i1}=\begin{bmatrix}
    A_{i}+B_{i}K_{i1}-L_{i}(R_{ii}+R_{ii}^{\t})C_{i} & B_{i}K_{i2} \\
    0_{p_{i}\times n_{i}} & 0_{p_{i}\times p_{i}}
  \end{bmatrix},~~
  M_{i2}=\begin{bmatrix}
           L_{i} \\
           I_{p_{i}}
         \end{bmatrix}.
\end{align*}
\end{corollary}

\bremark (\emph{Integrator in the strategy})
  When the agents are not affected by external disturbances, we can design the $p_{i}$-copy internal model $(G_{i1},G_{i2})$ as $(0_{p_{i}\times p_{i}},I_{p_{i}})$ for each agent $i$, $i\in\mathcal{I}$. Note that in this case, it is still necessary to include an integrator in the controller. This is because the steady-state output $\mathbf{y}$ satisfying $\bar{R}\mathbf{y}+\bar{Q}=0$ depends on the constant matrix $\bar{Q}$. If the agents have dynamics  (\ref{playerdynamics_i}) and the disturbance $w_{i}$ is a constant vector, the NE seeking strategy can use the same design as shown in Corollary \ref{corollary1}, as $\text{spec}(\widetilde{S}_{i})=\{0\}$ still holds. 
\eremark

\subsection{General communication graph}

Assumption~\ref{assumption_graph} can be relaxed if the distributed game strategy is designed as
\begin{align}\label{controller_distributed2}
\dot{\xi}_{i}&=A_{i}\xi_{i}+B_{i}u_{i}-L_{i}(\hat{e}_{i}-e_{i}),\notag\\
\dot{\zeta}_{i}&=G_{i1}\zeta_{i}+G_{i2}e_{i},\notag\\
u_{i} &= K_{i1}\xi_{i}+K_{i2}\zeta_{i}
\end{align}
where $\hat{e}_{i}=(R_{ii}+R_{ii}^{\t})C_{i}\xi_{i}+\sum_{j\in\mathcal{N}_{i}}R_{ij}C_{j}\xi_{j}$, matrix $L_{i}$ and the pair $(G_{i1},G_{i2})$ have the same definitions as in (\ref{controller_distributed1}). The difference between (\ref{controller_distributed2}) and (\ref{controller_distributed1}) is that in (\ref{controller_distributed2}) agents exchange $C_{i}\xi_{i}$ with neighbors. To rule out the case where the network contains isolated agents solving an optimization problem instead of playing games with neighbors, we have the following assumption on the communication graph.

\begin{assumption}\label{assumption_graph2}
  The communication graph among the agents is connected.
\end{assumption}

  Under Assumption~\ref{assumption_graph2}, the network topology can be a connected undirected graph, or a weakly or strongly connected digraph.

\begin{theorem}[Distributed strategy under connected communication graphs]\label{theorem2}
  Under Assumptions~\ref{assumption_stabilizability&detectability}, \ref{assumption_minimumphase}, and \ref{assumption_graph2}, distributed strategy (\ref{controller_distributed2}) is a solution to the NE seeking Problem~\ref{problem_NEseeking}.
\end{theorem}

\begin{proof}
Denoting $\xi=\text{col}(\xi_{1},\dots,\xi_{N})$, $\zeta=\text{col}(\zeta_{1},\dots,\zeta_{N})$, and $z=\text{col}(x,\xi,\zeta)$ gives the system matrix of the nominal closed-loop system as
\begin{align*}
A_{c}=\begin{bmatrix}
\bar{A} & \bar{B}\bar{K}_{1} & \bar{B}\bar{K}_{2}\\
\bar{L}\bar{R}\bar{C} & \bar{A}+\bar{B}\bar{K}_{1}-\bar{L}\bar{R}\bar{C} & \bar{B}\bar{K}_{2}\\
\bar{G}_{2}\bar{R}\bar{C} & 0 & \bar{G}_{1}
\end{bmatrix}.
\end{align*}
Note that $A_{c}$ here is in the same form as $A_{ci}$ in the proof of Theorem~\ref{theorem_digraph}. Therefore, under the assumption that $\bar{R}$ is positive definite, it can be proved in the same fashion that $A_{c}$ is Hurwitz and the regulator equations have a unique solution. Then, applying \citet[Theorem 1.31]{Huang2004}, we can prove that there exists a dynamic output feedback controller in the form of (\ref{controller_distributed2}) that stabilizes the closed-loop system and drives the error $e$ to zero asymptotically. Hence, following Proposition~\ref{proposition}, (\ref{controller_distributed2}) also solves the NE seeking problem.
\end{proof}

  By allowing the agents to exchange the additional information $C_{i}\xi_{i}$ with neighbors, we can relax the assumption on the communication graph in Theorem~\ref{theorem_digraph}. In (\ref{controller_distributed2}), the $\xi_{i}$-subsystem can be seen as an observer for the state $x_{i}$ and the $\zeta_{i}$-subsystem is the internal model for the exosystem (\ref{exosystem_new}).
  As the agents exchange $C_{i}\xi_{i}$ instead of full state estimation $\xi_{i}$, similar to the arguments in Remark~\ref{remark_privacy} the agents maintain some privacy.

\medskip

Similar to the previous subsection, we can derive a corollary for a disturbance-free NE seeking problem under general communication graph from Theorem \ref{theorem2}.

\begin{corollary}
  Consider an NE seeking problem with agent dynamics (\ref{playerdynamics_i_nodisturbance}) and cost functions (\ref{localcostfunction}) for all $i\in\mathcal{I}$. Under Assumptions~\ref{assumption_stabilizability&detectability} and \ref{assumption_minimumphase}, there exists matrices $L_{i}$, $K_{i1}$ and $K_{i2}$, such that the NE seeking problem has a solution in the form of
\begin{align}\label{controller_nodisturbance}
\dot{\xi}_{i}&=A_{i}\xi_{i}+B_{i}u_{i}-L_{i}(\hat{e}_{i}-e_{i}),\notag\\
\dot{\zeta}_{i}&=e_{i},\notag\\
u_{i} &= K_{i1}\xi_{i}+K_{i2}\zeta_{i}.
\end{align}
\end{corollary}

\section{Simulation results}\label{section_simulation}
In this section, the proposed NE seeking strategies are applied to connectivity control of sensor networks studied in \citet{Stankovic2012}. The network is composed of mobile robot agents to be positioned at optimal sensing points while keeping good connections with selected neighboring agents. In this example, we consider mobile robots modelled by
\begin{align}
\dot{x}_{i1}=x_{i2},\quad \dot{x}_{i2}=-c_{i}x_{i2}+u_{i}+w_{i}, \quad i\in\mathcal{I}
\end{align}
where $x_{i1}$ and $x_{i2}$ denote the position and velocity of each agent, respectively, $c_{i}>0$ is the friction parameter. The decision variable is the position of each agent, \emph{i.e.}, $y_{i}=x_{i1}$. The cost function is defined as
\begin{align}
  J_{i} = \|y_{i}-r_{i}\|^2+\sum_{j\in\mathcal{N}_{i}}\|y_{i}-y_{j}\|^2,
\end{align}
for each $i\in\mathcal{I}$, where $r_{i}$ is the objective position of each agent. Then, by converging to the NE, the agents compromise between the individual objective of moving to position $r_{i}$ and the collective objective of maintaining the connectivity with their neighbors.

In the simulation, the sensor network is composed of $5$ mobile robots subject to disturbance generated by the exosystem
\begin{align*}
 \dot{w}_{i}=\begin{bmatrix}
 0 & \pi/10 \\ -\pi/10 & 0
 \end{bmatrix}w_{i},\quad i\in\mathcal{I}
\end{align*}
which is a class of sinusoid signals having frequency $\pi/10$. The friction parameter is set as $c_{i}=0.2$. The initial positions of the agents are $(0,0)$, $(1,1)$, $(1,-1)$, $(2,1)$, $(2,-1)$, respectively.

\begin{figure}
  \centering
  \includegraphics[width=0.2\textwidth]{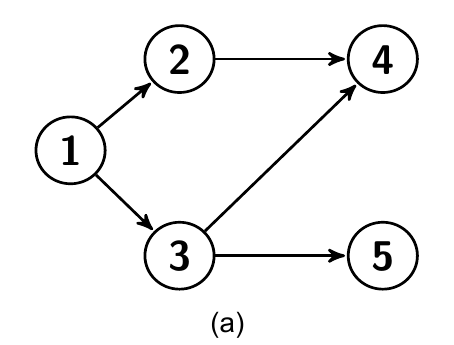}
  \includegraphics[width=0.2\textwidth]{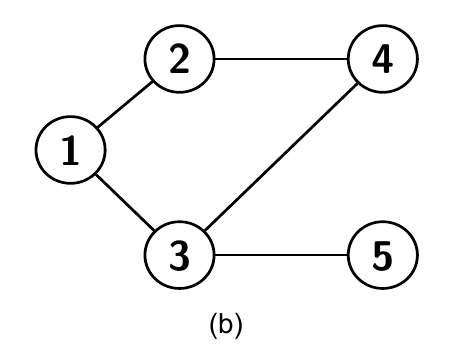}
  \caption{The information graphs of the sensor networks.}\label{figure_graph}
\end{figure}

First, we consider a network connected by a digraph illustrated in (a) of Figure~\ref{figure_graph}. The objective positions for the agents are set at $r_{1}=[-1~0]^{\t}$, $r_{2}=[1~-1]^{\t}$, $r_{3}=[2~-1]^{\t}$, $r_{4}=[-1~2]^{\t}$, and $r_{5}=[-2~2]^{\t}$. To keep connectivity with their neighbors, the agents will converge to the NE $y^{*}=[-1~0~0~-0.5~0.5~-0.5~-0.166~0.333~-0.75~-1.25]^{\t}$. Applying the distributed strategy (\ref{controller_distributed1}) in Theorem~\ref{theorem_digraph}, the robot agents can converge to the NE as illustrated in Figures~\ref{figure_trajectories1} and \ref{figure_errors1}. In Figure \ref{figure_trajectories1}, the initial positions of the agents are denoted by circles and the NE is denoted by a collection of crosses. Figure~\ref{figure_errors1} illustrates the regulated errors $e_{i}$ of the agents in $x$ and $y$ axes.

\begin{figure}
  \centering
  \includegraphics[width=0.45\textwidth]{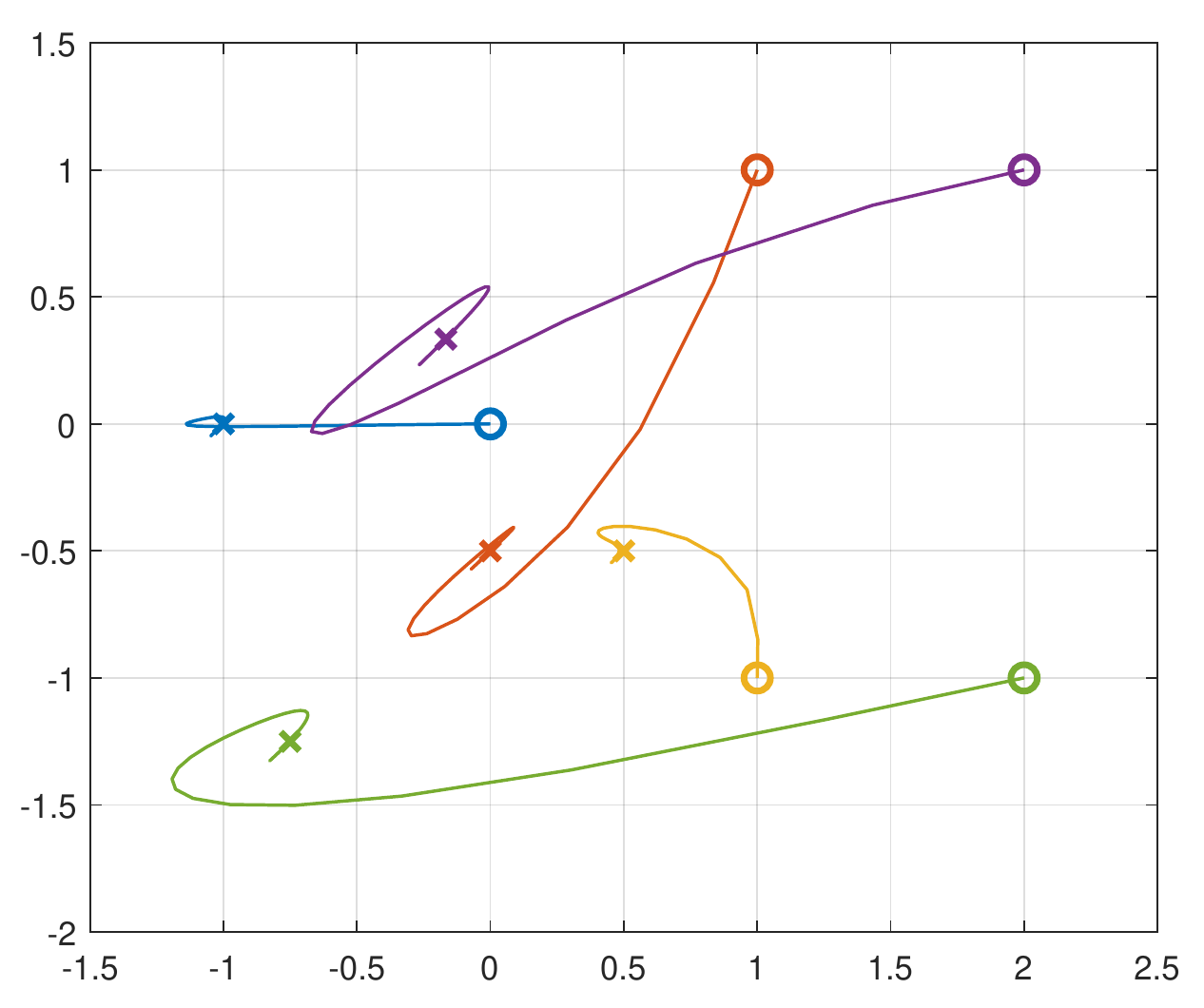}
  \caption{Trajectories of the agents under digraph.}\label{figure_trajectories1}
\end{figure}

\begin{figure}
  \centering
  \includegraphics[width=0.45\textwidth]{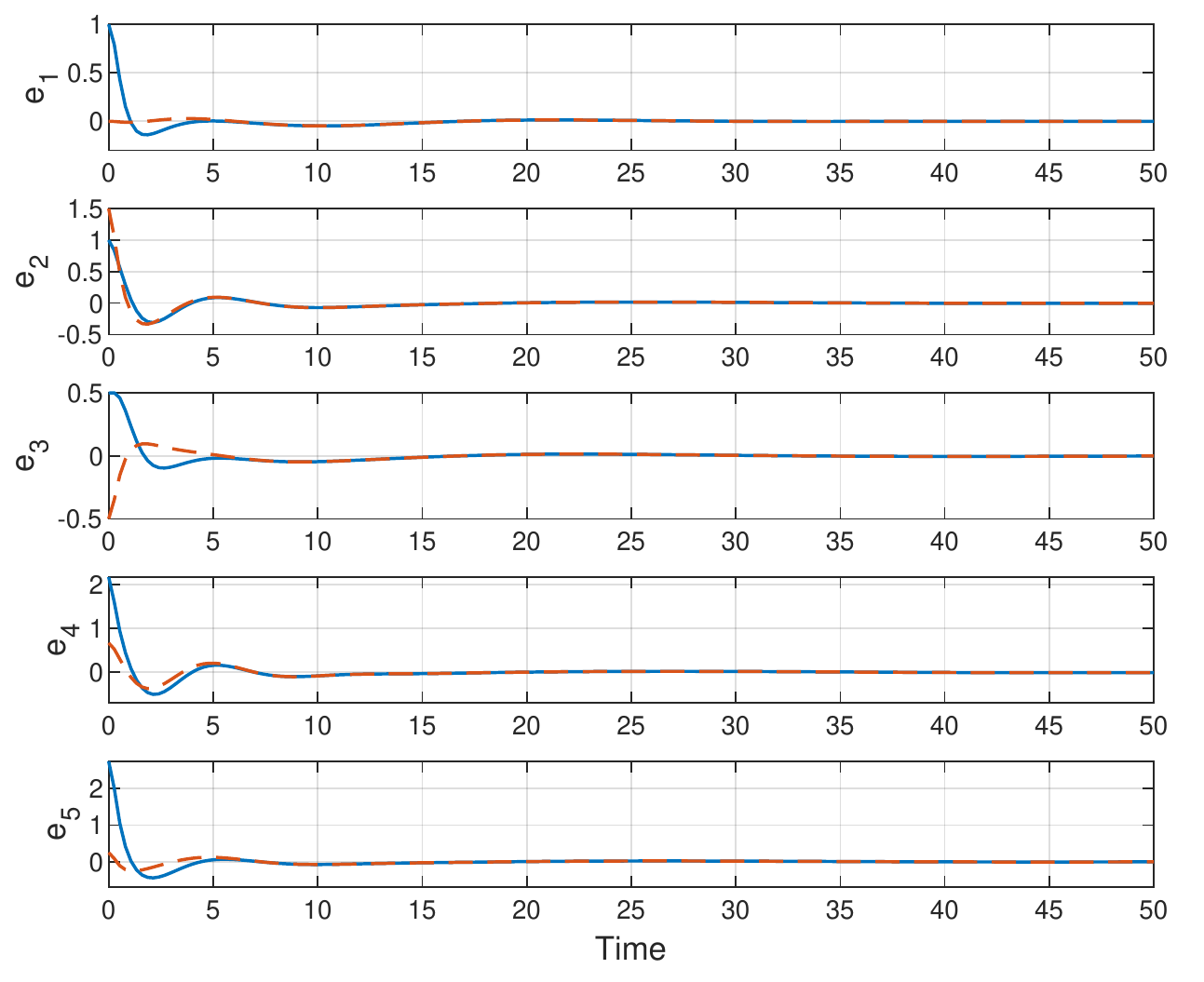}
  \caption{Errors of the agents under digraph.}\label{figure_errors1}
\end{figure}

The second case is when the underlying graph of the agents is undirected as illustrated in Figure~\ref{figure_graph}(b). The objective positions for the agents are set as the same in the previous case, and then the NE becomes $y^{*}=[-0.459~-0.133~0.204~-0.527~0.580~-0.130~-0.071~0.449~-0.709~-1.064]^{\t}$. The simulation results of applying distributed strategy (\ref{controller_distributed2}) to the NE seeking problem are illustrated in Figures~\ref{figure_trajectories2} and \ref{figure_errors2}.

\begin{figure}
  \centering
  \includegraphics[width=0.45\textwidth]{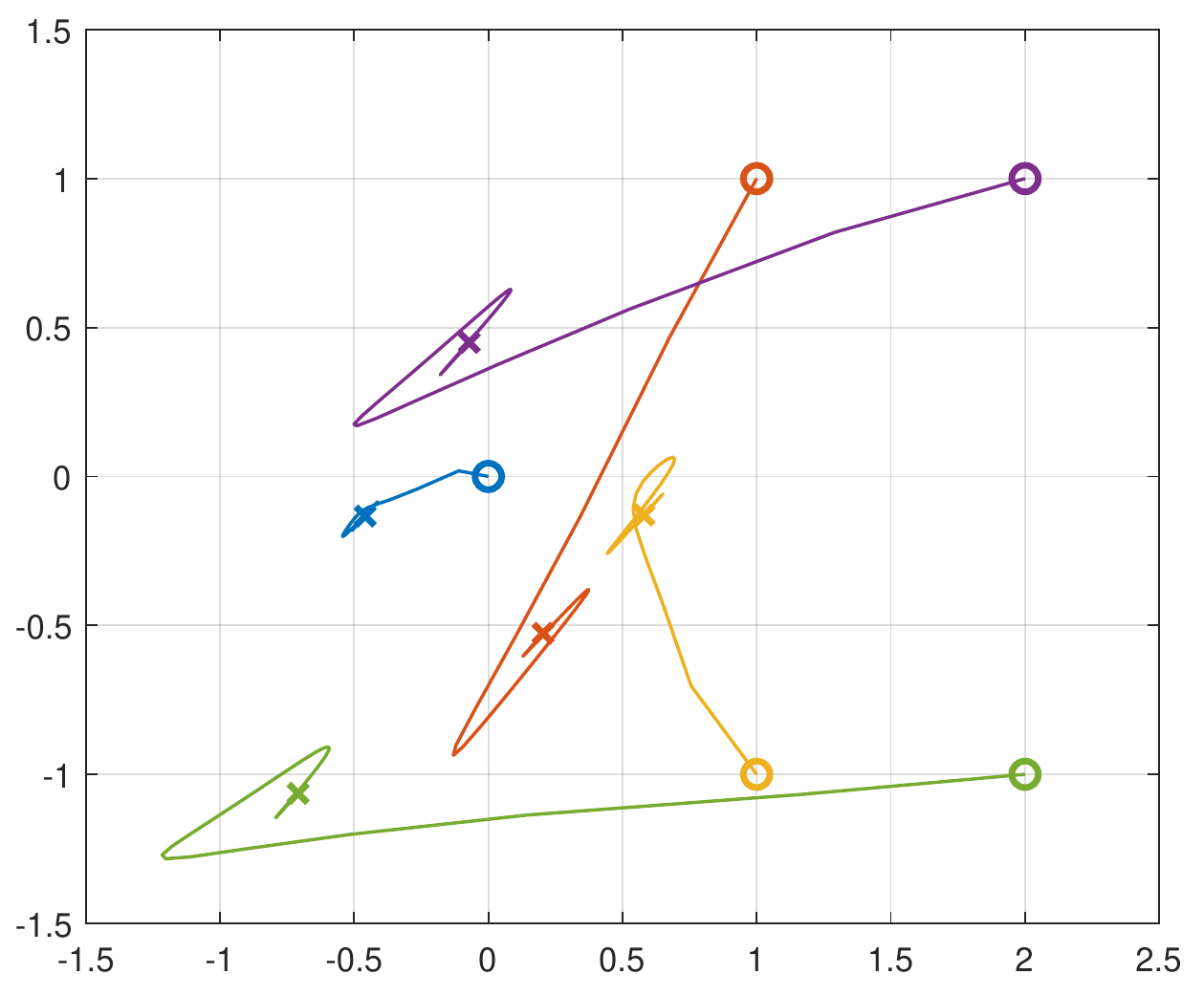}
  \caption{Trajectories of the agents under undirected graph.}\label{figure_trajectories2}
\end{figure}

\begin{figure}
  \centering
  \includegraphics[width=0.45\textwidth]{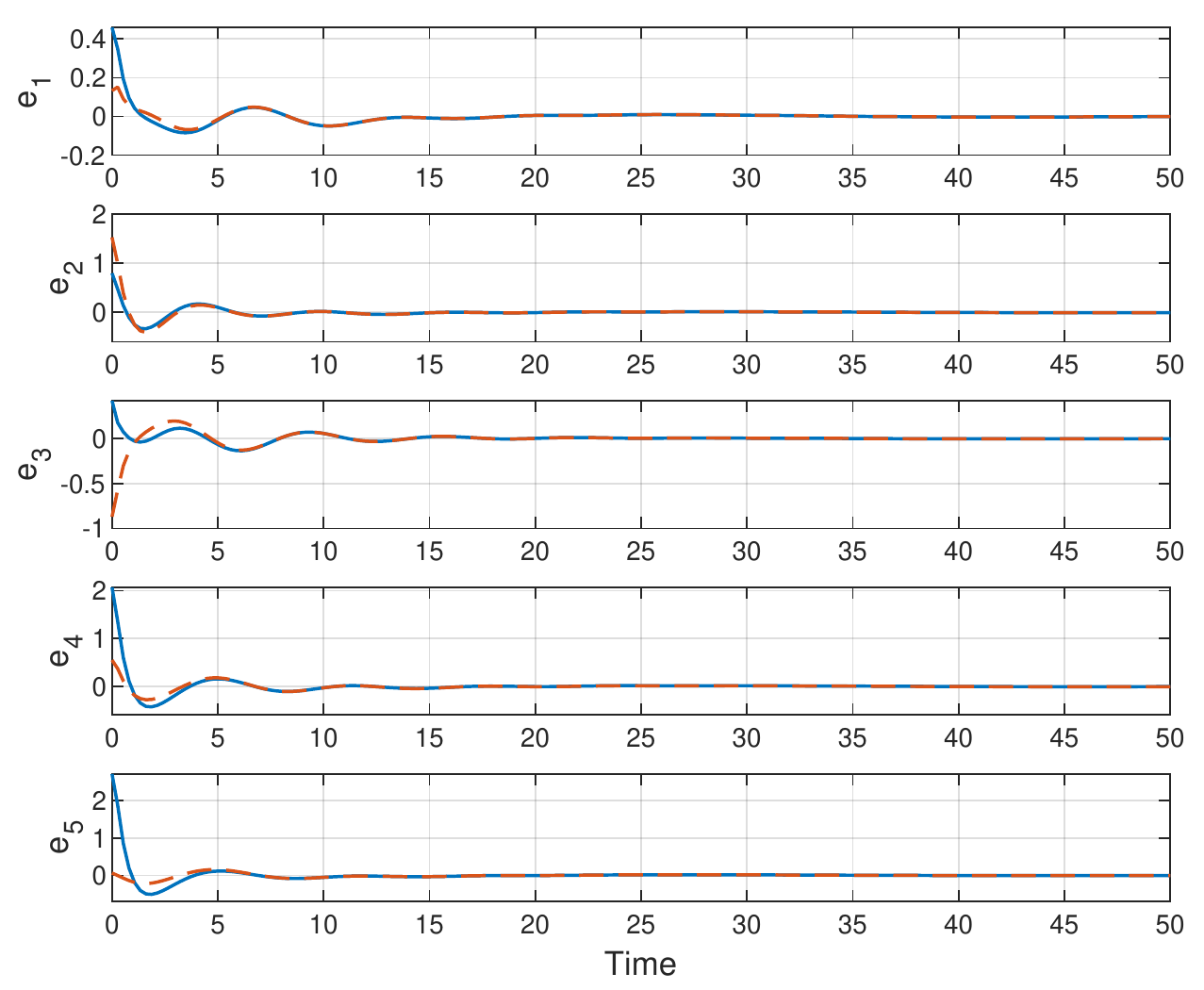}
  \caption{Errors of the agents under undirected graph.}\label{figure_errors2}
\end{figure}

\section{Conclusion}\label{section_conclusion}
Motivated by engineering applications, this paper considers output NE seeking problems in a class of network games with uncertain linear agent dynamics subject to external disturbances. The main challenge is to design a strategy capable of both driving the agent outputs to the NE and stabilizing the closed-loop dynamics. Other difficulties stem from the uncertainties in the dynamics, the external disturbance, and the relaxed assumptions on the communication graphs. By over coming these difficulties, our proposed game strategies have a wider range of  applications to engineering multi-agent problems. Future work may consider similar output network games with general cost functions and/or nonlinear agent dynamics.

\section{Appendix}
\begin{appendices}
\emph{Some insights regarding Assumption~\ref{assumption_minimumphase}}

Under Assumption~\ref{assumption_minimumphase}, it can be proved that for any nonsingular matrix $D_{i}\in\mathbb{R}^{p_{i}\times{p_{i}}}$ and any $\lambda\in\text{spec}(S_{i})\bigcup\{0\}$,
  \begin{align*}
    \text{rank}
    \begin{bmatrix}
      A_{i}-\lambda I & B_{i} \\
      D_{i}C_{i} & 0
    \end{bmatrix}=n_{i}+p_{i}.
  \end{align*}
To prove this, we denote
  \begin{align*}
    H_{i}=\begin{bmatrix}
        A_{i}-\lambda I & B_{i} \\
        C_{i} & 0
      \end{bmatrix},\quad
    F_{i}=\begin{bmatrix}
        I_{n_{i}} & 0_{n_{i}\times p_{i}} \\
        0_{p_{i}\times n_{i}} & D_{i}
      \end{bmatrix}.
  \end{align*}
Then, $F_{i}\in\mathbb{R}^{(n_{i}+p_{i})\times(n_{i}+p_{i})}$ is nonsingular, $\text{rank}~F_{i}=n_{i}+p_{i}$, $H_{i}\in\mathbb{R}^{(n_{i}+p_{i})\times(n_{i}+p_{i})}$, and
\begin{align*}
  F_{i}H_{i}=\begin{bmatrix}
       A_{i}-\lambda I & B_{i} \\
      D_{i}C_{i} & 0
     \end{bmatrix}.
\end{align*}
By \citet[Corollary 2.5.10]{Berstein2009}, the following inequality holds
\begin{align*}
  \text{rank}~F_{i}+\text{rank}~H_{i}-(n_{i}+p_{i}) & \le\text{rank}~F_{i}H_{i} \\
  & \le \text{min}\{\text{rank}~F_{i},\text{rank}~H_{i}\},
\end{align*}
which gives $n_{i}+p_{i}\le\text{rank}~F_{i}H_{i}\le n_{i}+p_{i}$, and thus $\text{rank}~F_{i}H_{i}=n_{i}+p_{i}$.

Moreover, under  Assumption~\ref{assumption_minimumphase}, for all $\lambda\in\text{spec}({\widetilde{S}})$,
  \begin{align*}
    \text{rank}
    \begin{bmatrix}
      \bar{A}-\lambda I & \bar{B} \\
      \bar{C} & 0
    \end{bmatrix}=\bar{n}+\bar{p},
  \end{align*}
which leads to
  \begin{align*}
    \text{rank}
    \begin{bmatrix}
      \bar{A}-\lambda I & \bar{B} \\
      \bar{R}\bar{C} & 0
    \end{bmatrix}=\bar{n}+\bar{p}.
  \end{align*}

\end{appendices}

\bibliographystyle{plainnat}
\bibliography{referenceGame}

\begin{thebibliography}{19}
\providecommand{\natexlab}[1]{#1}
\providecommand{\url}[1]{\texttt{#1}}
\expandafter\ifx\csname urlstyle\endcsname\relax
  \providecommand{\doi}[1]{doi: #1}\else
  \providecommand{\doi}{doi: \begingroup \urlstyle{rm}\Url}\fi

\bibitem[Ba{\c s}ar and Olsder(1999)]{Bacsar1999Book}
Tamer Ba{\c s}ar and Geert~Jan Olsder.
\newblock \emph{Dynamic Noncooperative Game Theory}.
\newblock SIAM, 2nd edition, 1999.

\bibitem[Bernstein(2009)]{Berstein2009}
Dennis~S. Bernstein.
\newblock \emph{Matrix Mathematics Theory, Facts, and Formulas}.
\newblock Princeton University Press, 2nd edition, 2009.

\bibitem[{De Persis} and Grammatico(2019)]{DePersis2019Auto_NEseeking}
Claudio {De Persis} and Sergio Grammatico.
\newblock Distributed averaging integral {N}ash equilibrium seeking on
  networks.
\newblock \emph{Automatica}, 110:\penalty0 108548(1)--108548(7), December 2019.

\bibitem[Deng and Liang(2019)]{Deng2019Auto_EL}
Zhenhua Deng and Shu Liang.
\newblock Distributed algorithms for aggregative games of multiple
  heterogeneous {E}uler-{L}agrange systems.
\newblock \emph{Automatica}, 99:\penalty0 246--252, January 2019.

\bibitem[Facchinei and Pang(2003)]{Facchinei2003}
Francisco Facchinei and Jong-Shi Pang.
\newblock \emph{Finite-Dimensional Variational Inequalities and Complementarity
  Problems}.
\newblock Springer-Verlag, New York, 2003.

\bibitem[Gadjov and Pavel(2019)]{Godjov2018TAC}
Dian Gadjov and Lacra Pavel.
\newblock A passivity-based approach to {N}ash equilibrium seeking over
  networks.
\newblock \emph{IEEE Transactions on Automatic Control}, 64\penalty0
  (3):\penalty0 1077--1092, March 2019.

\bibitem[Grammatico(2018)]{Grammatico2018TCNS_networkgames}
Sergio Grammatico.
\newblock Proximal dynamics in multiagent network games.
\newblock \emph{IEEE Transactions on Control of Network Systems}, 5\penalty0
  (4):\penalty0 1707--1706, December 2018.

\bibitem[Huang(2004)]{Huang2004}
Jie Huang.
\newblock \emph{Nonlinear Output Regulation: Theory and Applications}.
\newblock SIAM, Philadelphia, 2004.

\bibitem[Isidori et~al.(2003)Isidori, Marconi, and Serrani]{Isidori2003book}
Alberto Isidori, Lorenzon Marconi, and Andrea Serrani.
\newblock \emph{Robust Autonomous Guidance: An Internal Model Approach}.
\newblock Springer-Verlag London, 2003.

\bibitem[Koshal et~al.(2016)Koshal, Nedi{\' c}, and Shanbhag]{Koshal2016}
Jayash Koshal, Angelia Nedi{\' c}, and Uday~V. Shanbhag.
\newblock Distributed algorithms for aggregative games on graphs.
\newblock \emph{Operations Research}, 64\penalty0 (3):\penalty0 680--704, May
  2016.

\bibitem[Lin et~al.(2014)Lin, Qu, and Simaan]{Lin2014CDC_formation}
Wei Lin, Zhihua Qu, and Marwan~A. Simaan.
\newblock Distributed game strategy design with application to multi-agent
  formation control.
\newblock In \emph{53rd IEEE Conference on Decision and Control}, pages
  433--438, Los Angeles, California, USA, December 2014.

\bibitem[Parise et~al.(2015)Parise, Gentile, Grammatico, and
  Lygeros]{Parise2015_CDCnetworkgames}
Francesca Parise, Basilio Gentile, Sergio Grammatico, and John Lygeros.
\newblock Network aggregative games: {D}istributed convergence to {N}ash
  equilibria.
\newblock In \emph{54th IEEE Conference on Decision and Control}, pages
  2295--2300, Osaka, Japan, December 2015.

\bibitem[Pavel(2006)]{Pavel2006_opticalnetwork}
Lacra Pavel.
\newblock A noncooperative game approach to {OSNR} optimization in optical
  networks.
\newblock \emph{IEEE Transactions on Automatic Control}, 51\penalty0
  (5):\penalty0 848--852, May 2006.

\bibitem[Romano and Pavel(2019{\natexlab{a}})]{Romano}
Andrew~R. Romano and Lacra Pavel.
\newblock Dynamic {NE} seeking for multi-integrator networked agents with
  disturbance rejection.
\newblock arXiv preprint arXiv:1903.02587v1, 2019{\natexlab{a}}.

\bibitem[Romano and Pavel(2019{\natexlab{b}})]{Romano2019ECC}
Andrew~R. Romano and Lacra Pavel.
\newblock Dynamic {N}ash equilibrium seeking for higher-order integrators in
  networks.
\newblock In \emph{18th European Control Conference}, pages 1029--1035, Napoli,
  Italy, June 25-28, 2019{\natexlab{b}}.

\bibitem[Salehisadaghiani and Pavel(2016)]{Salehisadaghiani2016}
Farzad Salehisadaghiani and Lacra Pavel.
\newblock Distributed {N}ash equilibrium seeking: {A} gossip-based algorithm.
\newblock \emph{Automatica}, 72:\penalty0 209--216, October 2016.

\bibitem[Stankovi{\' c} et~al.(2012)Stankovi{\' c}, Johansson, and Stipanovi{\'
  c}]{Stankovic2012}
M.~S. Stankovi{\' c}, K.~H. Johansson, and D.~M. Stipanovi{\' c}.
\newblock Distributed seeking of {N}ash equilibria with application to mobile
  sensor networks.
\newblock \emph{IEEE Transactions on Automatic Control}, 57\penalty0
  (4):\penalty0 904--919, April 2012.

\bibitem[Wang et~al.(2014)Wang, Saad, Han, Poor, and Ba{\c
  s}ar]{Wang2014game_smartgrid}
Yunpeng Wang, Walid Saad, Zhu Han, H.~Vincent Poor, and Tamer Ba{\c s}ar.
\newblock A game-theoretic approach to energy trading in the smart grid.
\newblock \emph{IEEE Transactions on Smart Grid}, 5\penalty0 (3):\penalty0
  1439--1450, May 2014.

\bibitem[Ye and Hu(2017)]{Ye2017TAC_NEseeking}
Maojiao Ye and Guoqiang Hu.
\newblock Distributed {N}ash equilibrium seeking by a consensus based approach.
\newblock \emph{IEEE Transactions on Automatic Control}, 62\penalty0
  (9):\penalty0 4811--4818, September 2017.

\end{thebibliography}
\end{document}